\documentclass[a4paper,11pt,reqno]{amsart}

\pdfoutput=1

\usepackage{a4wide,color,array}
\usepackage{cite}
\usepackage{hyperref}
\usepackage{amsmath,amssymb,amsthm,color}
\hypersetup{linktocpage,colorlinks, citecolor={magenta}}
\usepackage[english]{babel}
\usepackage[utf8]{inputenc}

\usepackage{xspace}
\usepackage{bm}				
\usepackage{bbm,upgreek}		
\usepackage[e]{esvect}		
\usepackage{esint}			
\usepackage{etoolbox}			
\usepackage{calc}				
\usepackage{eso-pic}			
\usepackage{datetime}			

\usepackage{lmodern}
\numberwithin{equation}{section}

\usepackage{enumitem}
\setlist{nosep}
\setlist{noitemsep}

\newcommand{\Q}{\mathbb{Q}}
\newcommand{\R}{\mathbb{R}}

\newcommand{\I}{\mathcal{I}}

\newtheorem{theorem}{Theorem}
\newtheorem{proposition}{Proposition}
\newtheorem{lemma}{Lemma}

\theoremstyle{plain}

\theoremstyle{definition}

\newcommand{\tref}[1]{Theorem~\ref{t.#1}}
\newcommand{\pref}[1]{Proposition~\ref{p.#1}}
\newcommand{\lref}[1]{Lemma~\ref{l.#1}}
\newcommand{\cref}[1]{Corollary~\ref{c.#1}}

\newcommand{\eref}[1]{(\ref{e.#1})}

\def \1{\mathbf{1}} 
\def \mcl{\mathcal}
\def \mbb{\mathbb}
\def \ep{\varepsilon}
\def \dist{\mathrm{dist}}
\newcommand{\g}{\mathsf{g}}

\def\nab{\nabla}

\def\({\left(}
\def\){\right)}

\def\XXint#1#2#3{{\setbox0=\hbox{$#1{#2#3}{\int}$}
		\vcenter{\hbox{$#2#3$}}\kern-.5\wd0}}

\newcommand{\Iso}{\mathrm{Iso}}
\newcommand{\Mim}{\mathrm{Mim}}

\def\nab{\nabla}
\def\pa{\partial}

\renewcommand{\P}{\mathbb{P}}

\makeatletter
\def\namedlabel#1#2{\begingroup
	#2%
	\def\@currentlabel{#2}%
	\phantomsection\label{#1}\endgroup
}
\makeatother

\begin{document}
	\title[Non-rigidity of the Coulomb gas]{Non-rigidity Properties of the Coulomb gas}
	\author{Eric Thoma}
	\date{March 6, 2023}
	\subjclass[2020]{60G55, 82B05, 60D05}
	
	\begin{abstract}
		We prove existence of infinite volume $d$-dimensional Coulomb gases which are not number rigid for $d \geq 3$. This makes the Coulomb gas the Gibbs point process with the longest range pairwise interaction (i.e.\ with the smallest $s$ in the interaction kernel $\g(x) = |x|^{-s}$) for which number non-rigidity has been proved in $d \geq 3$. We rule out properties stronger than number rigidity for the two-dimensional Coulomb gas.
	\end{abstract}
	
	\maketitle
	
	\section{Introduction} \label{s.intro}
	The subjects of this article are point processes $X$ on $\R^d$, $d \geq 2$, with law $\P$ arising as weak limits of microscopic point processes associated to $d$-dimensional Coulomb gases. More precisely, we consider Gibbs measures $\P_N$ on point configurations $X_N = (x_1,\ldots,x_N) \in (\R^d)^N$ defined by
	\begin{equation} \label{e.PNdef}
		\P_N(dX_N) \propto e^{-\beta \mcl H(X_N)} dX_N, \quad \mcl H(X_N) := \frac12 \sum_{\substack{i,j = 1 \\ i \ne j}}^N \g(x_i - x_j) + \sum_{i=1}^N W(x_i)
	\end{equation}
	for the Coulomb kernel $\g$ given by $x \mapsto -\log|x|$ in $d=2$ and $x \mapsto |x|^{-d+2}$ in $d\geq 3$, a confining potential $W : \R^d \to \R$, possibly $N$ dependent, and an inverse temperature $\beta > 0$. We can associate to $X_N$ the Radon measure $X = \sum_{i=1}^N \delta_{x_i}$, and we endow Radon measures with the topology generated by the dual pairing with continuous, compactly support functions $\R^d \to \R$. Weak subsequential limits of the law of $X$ (or versions re-centered around a sequence of points $z_N$) as $N \to \infty$ are known to exist in some generality \cite{AS21,T23}. Such limits are {\it thermodynamic limits}, and we call any limit point $\P$ an {\it infinite volume Coulomb gas}. We will assume throughout that $\P$ is an infinite volume Coulomb gas with a few conditions to be stated later. All point processes will be simple and locally finite.

	A point process $X$ is said to be {\it number rigid} if for any bounded Borel set $\Omega$, there is a Borel measurable function $F$ with $X(\Omega) = F(X_{|\Omega^c})$ almost surely. Here $X_{|\Omega^c}$ denotes the restriction of $X$ to the complement of $\Omega$. In other words, the number of points within $\Omega$ is determined a.s.\ by the point configuration outside of $\Omega$. We now state our first main result.
	\begin{theorem} \label{t.main}
		For any $d \geq 3$, there exists an infinite volume Coulomb gas which is not number rigid.
	\end{theorem}
	
	While \tref{main} only claims the existence of one gas without number rigidity, any bounded density, non-vacuum $\P$ which is a limit of finite gases (or satisfies canonical DLR conditions in the sense of \cite{DHLM21}) is number non-rigid by \pref{main} (or a straightforward modification of \pref{main}).
	
	
	The (number) rigidity property, first introduced in by Ghosh and Peres in \cite{GP17}, has received significant attention for a variety of point processes in recent years. It is related to a previous notion of a hereditary process in stochastic geometry (see \cite{D17} and references therein). There have thus far been two major approaches toward proving rigidity. The method of \cite{GP17} involves estimates for fluctuations of certain linear statistics and was first used to prove rigidity for the Ginibre ensemble (which is a certain 2D infinite volume Coulomb gas) and the zeros of a certain Gaussian analytic function. In \cite{RN18}, a similar method was used to prove rigidity of the $\mathrm{Sine}_\beta$ process, which roughly corresponds to $d=1$ and $\g(x) = -\log|x|$ in \eref{PNdef} (in infinite volume and with proper scaling on $W$) and which is especially relevant to the present article because it has a logarithmic pair interaction, just like the two-dimensional Coulomb gas.
	
	A second approach based on arguments from statistical mechanics and the DLR formalism was introduced in \cite{DHLM21}, where it was put to use for a proof of rigidity for the $\mathrm{Sine}_\beta$ process, independently of \cite{RN18}. Later, number non-rigidity for circular Riesz gases with parameter $s \in (d-1,d)$ and $d \geq 1$ was shown in \cite{DV21} using the DLR approach (and the transport of particles to ``$\infty$"). The Riesz gas corresponds to \eref{PNdef} with $\g(x) = |x|^{-s}$, which is notably longe-range and non-integrable for $s < d$, and so \cite{DV21} eliminated a natural conjecture for a possible threshhold in $s$ between rigidity and non-rigidity (and proposed a new conjectured threshhold of $s = d-1$, a conjecture which we disprove).
		
	Regarding the Coulomb gas, the Ginibre ensemble is an infinite volume Coulomb gas in $d=2$ with $\beta=2$ and certain quadratic $W$, so the 2D Coulomb gas is rigid for at least this temperature and potential. Furthermore, if the canonical DLR equations are established for the two-dimensional gas, it will establish number rigidity (see \cite[Theorem 3.18]{DHLM21}). The present article proves non-rigidity for the gas in $d \geq 3$, which also represents the longest range Riesz gas yet proved non-rigid. Notably, (non-)rigidity of the Riesz gas with $s \in (d-2,d-1]$ remains open in all dimensions. We expect the methods of this article will be useful in studying these cases (in particular, integrability at $\infty$ of $\nab^2 \g$ should allow our general method to work, at least at a non-rigorous heuristic level).
	
	Rigidity for other classes of point processes has been considered in several works. Perturbed lattices were considered in \cite{PS14} and various determinantal point processes in \cite{G15,B16,BQ17}.
	
	Our proof technique is based on transporting particles within the gas while controlling entropy and energy costs associated to the movement. Two novelties specific to the present article are the {\it symmetric} movement of particles over {\it arbitrarily large distances} and a fundamental need to work in {\it finite volume} and {\it finite particle count}. Indeed, we crucially use an approximation in which we condition on $X_{|B_S \setminus \Omega}$ for a large ball $B_S$ of radius $S$ rather than on $X_{|\Omega^c}$, and we transport particles between $\Omega$ and $B_S^c$. We achieve estimates for various transport costs that are independent of $S$ by a symmetrization trick: if it is unfavorable to transport from $\Omega$ to $B_S^c$, then it must be favorable to transport from $B_S^c$ to $\Omega$. We also use the finite particle number approximations $\P_N$ to $\P$ since entropy factors associated to the indistinguishability and labeling of particles will play an important role. This feature is somewhat surprising given that number rigidity is a genuinely {\it infinite-volume} phenomenon (for example, all $N$-particle processes are trivially number rigid). We use neither any DLR formalism nor any estimates for gas fluctuations beyond some rough bounds on particle density. Our ideas have strong ties to the isotropic averaging techniques developed in \cite{T23}, especially the technique of ``mimicry" used there to prove tightness of the properly rescaled minimal particle gap in an $N$-particle Coulomb gas.
	
	We will deduce \tref{main} from the following proposition. We denote by $B_r(z)$ the open ball of radius $r$ centered at $z$. Throughout the paper, we fix a distinguished center point $z_0$ and write $B_r$ for $B_r(z_0)$. Whenever we condition on a (Borel measurable) random variable, we mean the regular distribution conditioned on the $\sigma$-algebra generated by the random variable.
	\begin{proposition} \label{p.main}
		Let $\delta > 0$ and let $X$ be a point process on $\R^d$, $d \geq 3$, with law $\P$ having following properties:
		\begin{enumerate}
			\item (Non-vacuum) We have
			\begin{equation} \label{e.nonzerodensity} \limsup_{T \to \infty} \P(\{X(B_{2T} \setminus B_{T}) \geq \delta T^d \}) = 1. \end{equation} 
			\item (Limit of Coulomb gases) We have $\P = \lim_{N \to \infty} \P_N$ weakly along a subsequence of $N$-particle Coulomb gases $\P_N$ with $\Delta W \leq \delta^{-1}$ for confining potentials $W = W_N$ and inverse temperatures $\delta \leq \beta = \beta_N \leq \delta^{-1}$.
		\end{enumerate}
		Then for any $\ep > 0$, there exists an $R > 0$ and an $\sigma(X_{|B_R^c})$-measurable event $G_R$ with $\P(G_R) > 1- \ep$ and
		\begin{align} \label{e.main}
			\lefteqn{ \P\(\{X(B_R) = n\} \ \big | \ X_{|B_R^c}\) } \quad & \\ \notag &\leq \ep + C \P\(\{X(B_R) = n-1\} \ \big | \ X_{|B_R^c}\) + C \P\(\{X(B_R) = n+1\} \ \big | \ X_{|B_R^c}\) \quad \forall n \geq 0
		\end{align}
		$\P$-almost surely on $G_R$, where $C = C(\ep, \delta)$.
	\end{proposition}
	\begin{proof}[Proof of \tref{main}]
		Clearly any measure $\P$ in the above proposition is not number rigid. Thus it only remains to verify the existence of such a measure to prove \tref{main}, but this is a clear consequence of \cite[Theorem 1]{AS21}. To be precise, \cite{AS21} (and also \cite{T23}) prove that the finite dimensional distribution $(X(A_1),\ldots,X(A_k))$ is tight for bounded Borel sets $A_1,\ldots,A_k$ under some conditions for sequences of Coulomb gases $\P_N$. For point processes, convergence of the finite dimensional distributions and weak convergence are equivalent (see \cite[Theorem 11.1.VII]{DVJ08}). The non-vacuum condition \eref{nonzerodensity} is satisfied for limit points $\P$ under certain conditions by the quantitative lower bounds for density on large microscopic scales of \cite{AS21} or \cite[Theorem 6]{T23}.
	\end{proof}
	
	Some point processes exhibit rigidity properties even stronger than number rigidity. For example, the translation-invariant zero process of a Gaussian analytic function has both the number and center of mass of the points within a bounded Borel set $\Omega$ determined a.s.\ by the configuration outside $\Omega$. In other words, it is both number and center-of-mass rigid. An infinite family of different rigidity properties with associated processes are found in \cite{GK21}.
	
	Ghosh and Peres also introduced the concept of {\it tolerance} (under rigidities). Roughly speaking, a point process is tolerant if, conditional on the external configuration, it can assume any configuration on the manifold determined by a set of rigidity invariants. For example, a number-rigid process $X$ is tolerant if the law of $X$ restricted to a ball $B$ conditional on $X_{|B^c}$ induces a probability measure on configurations in $B^{X(B)}$ which is a.s.\ mutually absolutely continuous with Lebesgue measure. \cite{GP17} prove this property for the Ginibre ensemble. The proof involves estimates on inverse power sums over the point process with delicate cancellations. \cite{DHLM21} accomplishes a similar task using the DLR formalism for the $\mathrm{Sine}_\beta$ process (establishing DLR equations also involves special cancellations within the point process). Tolerance is explored in depth in \cite{GK21}.
	
	We remark that the $k$-point correlation function estimates of \cite{T23} apply to infinite volume Coulomb gases in any $d \geq 2$ (through convergence of the finite dimensional distributions of a finite gas). This implies, for example, that
	\begin{equation}
		\P\(\bigcap_{i=1}^n \{X(A_i) \geq 1\}\)  \leq C_{n,\delta} \prod_{i=1}^n \mathrm{vol}(A_i)
	\end{equation}
	for disjoint balls $A_i \subset \R^d$ and $\P$ any limit point of $\P_N$ fitting point (2) of \pref{main} in any $d \geq 2$. Unfortunately, this seems to be not enough to conclude a similar result for conditioned versions of $\P$. This can be fixed by applying the {\it isotropic averaging} argument of \cite{T23} directly to finite volume conditioned gases. Indeed, the argument is completely ``localized" in the gas, meaning it applies uniformly under conditioning that is sufficiently spatially separated from the events being bounded.
	\begin{theorem} \label{t.infvolrhon}
		Let $\delta > 0$ and $d \geq 2$. Consider an infinite volume $d$-dimensional Coulomb gas $\P$ which is a weak limit of finite volume gases $\P_N$ along a subsequence. Furthermore assume $\Delta W \leq \delta^{-1}$ and $\delta \leq \beta  \leq \delta^{-1}$.for the potentials $W = W_N$  and inverse temperatures $\beta = \beta_N$ of the gases $\P_N$. Then for any $R > 0$, $n \geq 1$, and disjoint balls $A_1,\ldots,A_n \subset B_R$, we have
		\begin{align} \label{e.infvolrhon}
			&{ \P\(\{X(A_1) \geq 1\} \cap \ldots \cap \{X(A_n) \geq 1 \} \ \big | \ X_{|B_R^c}, \{ X(B_R) = n \}\)} \\  \notag &\quad \leq \frac{C_{\delta,n,R} }{\P(\{X(B_R) = n\} \ | \ X_{|B_R^c})} \ep^{-nd} \prod_{i=1}^n \mathrm{vol}(A_i)
		\end{align}
		with probability $1$ under $\P$, where 
		$$
		\ep = \min_{i} \dist(A_i, \pa B_R).
		$$
	\end{theorem}

	Note that \eref{infvolrhon} holds $\P$-a.s.\ simultaneously for all disjoint balls $A_1,\ldots,A_n$ within $B_R$. Indeed, we can insist it holds for all balls with rational center and radii and then use that any ball is a decreasing limit of ones with rational data. \tref{infvolrhon}, along with the fact that $X$ is a.s.\ simple under $\P$, implies that the point process $X$ under $\P$ conditioned on $X_{|B_R^c}$ and on $X(B_R)$ has a law which, considered as a probability measure on $(B_R)^{X(B_R)}$, is a.s.\ absolutely continuous with respect to Lebesgue measure ({\it mutual} absolute continuity does not follow, so this is weaker than the full notion of  tolerance). Still, we may conclude, for example, that the two-dimensional Coulomb gas cannot be center-of-mass rigid. We have not attempted to make \tref{infvolrhon} as sharp or general as possible. One can for instance extract particle repulsion factors in the RHS of \eref{infvolrhon}, quantitative estimates on $R,n,\delta$ dependence, etc..

	\section{Proof of \pref{main}} \label{s.rigid}
	We first reduce the problem into a question about a certain finite particle Coulomb gas and define some of the central operators and events in the proof. Then we give a heuristic -- but essentially correct -- overview of the proof, followed by a fully rigorous proof.
	
	\subsection{First reductions and definitions.} \ Let $E_{n,R}$ be the event that $X$ has $n$ particles in the open disk $B_R$ of radius $R$. We have that
	$$
	\P\(E_{n,R} \ \big | \ X_{|B_R^c}\) = \lim_{S \to \infty} \P\(E_{n,R}\ \big | \ X_{|B_S \setminus B_R}\) 
	$$
	The limit can be understood in $L^2(\P)$. We will study $\P_N(E_{n,R} \ | \ X_{|B_S \setminus B_R})$ for large $N$ to deduce results about $ \P\(E_{n,R}\ \big | \ X_{|B_S \setminus B_R}\)$. It is unclear whether the law of the $\P_N(E_{n,R} \ | \ X_{|B_S \setminus B_R} )$, $n \geq 1$, under $\P_N$ will converge to that of $\P(E_{n,R} \ | \ X_{|B_S \setminus B_R} )$ under $\P$. We overcome this obstacle by proving quantitative inequalities for the conditioned $\P_N$ that do allow $N \to \infty$ limits under the assumption of weak convergence only.

	For a finite nonnegative measure $\mu$ supported within $B_S \setminus B_R$, we define
	\begin{equation} \label{e.Qdef}
		\mbb Q_{M,S,R}^{\mu}(dY_M) \propto \exp(-\beta \mcl H^{\mu}(Y_M)) \prod_{i =1}^M \1_{B_R \cup B_S^c}(y_i) dy_i,
	\end{equation}
	as a probability measure on point configurations $Y_M = (y_1,\ldots,y_M) \in (B_R \cup B_S^c)^M$ or, equivalently, the point process $Y = \sum_{i=1}^M \delta_{y_i}$, where
	\begin{equation} \label{e.Hmudef}
	\mcl H^{\mu}(Y_M) = \mcl H(Y_M) + \int_{B_S \setminus B_R} \sum_{i=1}^M \g(x - y_i) \mu(dx).
	\end{equation}
	Then $\mbb Q^{X_{|B_S \setminus B_R}}_{M,S,R}$ with $M = N - X(B_S \setminus B_R)$ is the law of $Y = X_{|B_R \cup B_S^c}$ under $\P_N$ conditioned on $X_{|B_S \setminus B_R}$. With the parameters $M,S,R, \mu$ fixed, we will abbreviate $\mbb Q = \mbb Q^{\mu}_{M,S,R}$. All implicit constants $C$ will be independent of $M$, $S$, and $\mu$, but may depend continuously on $R$ and $n$ (as well as $\beta$, $\beta^{-1}$, and $\sup \Delta W$). Later, we will introduce a third scale $T \gg S$ (we eventually take $T \to \infty$). 
	
	
	We will use ``mimicry" operators, first defined in \cite{T23}, to transport points in and out of $B_R$. Let $\nu$ be the uniform probability measure on $B_1(0) \subset \R^d$ (we will use $\nu$ to denote the density of $\nu$). For $i,j \in [M] := \{1,2,\ldots,M\}$, $i \ne j$, and any nice enough function $f : (\R^d)^2 \to \R$, we define $\Mim_{i \to j} f : (\R^d)^2 \to \R$ by
	\begin{equation} \label{e.mimdef}
		\Mim_{i \to j} f(y_i,y_j) = \int_{\R^d} f(y_j + x, y_j) \nu(dx).
	\end{equation}
	By convention, we let $\Mim_{i \to j}$ act on functions of $Y_M$ or any set of labeled coordinates by freezing all coordinates not labeled $i,j$ and applying \eref{mimdef}.
	
	\begin{lemma} \label{l.Mim}
		Let $d \geq 3$. For any finite, compactly supported, nonnegative measure $\mu$ and $Y_M \in (\R^d)^M$, we have
		\begin{equation} \label{e.Mim.energy}
			\min \( \Mim_{i \to j} \mcl H^\mu(Y_M), \Mim_{j\to i} \mcl H^\mu(Y_M) \) \leq \mcl H^\mu(Y_M) + C.
		\end{equation}
		For the adjoint operator $\Mim_{i \to j}^\ast$ and any bounded, compactly supported function $f$, we have
		\begin{equation} \label{e.Mim.adj}
			\Mim_{i \to j}^\ast f(y_i,y_j) = \nu(y_i - y_j) \int_{\R^d} f(x,y_j) dx.
		\end{equation}
	\end{lemma}
	\begin{proof}
		We break $\mcl H^\mu(Y_M)$ into parts:
		\begin{equation*}
			\mcl H^\mu(Y_M) =: \mcl H^\mu_{i,j, \mathrm{ext}}(Y_M) + \( \sum_{k \ne i,j} \g(x_i - x_k) + \g(x_j - x_k)\) + \g(x_i - x_j) + W(x_i) + W(x_j).
		\end{equation*}
		The term $\mcl H^\mu_{i,j, \mathrm{ext}}(Y_M)$ has no $y_i$ or $y_j$ dependence and so is invariant under the mimicry operators we consider. For the other pieces, we have
		\begin{equation*}
			\Mim_{i \to j} \g(x_i - x_k) \leq \g(x_j - x_k), \quad \Mim_{j \to i} \g(x_j - x_k) \leq \g(x_i - x_k)
		\end{equation*}
		since $\g$ is superharmonic, $\Mim_{i \to j} \g(x_i - x_j) \leq C$ since the singularity in $\g$ is integrable, and $\Mim_{i \to j} W(x_i) \leq W(x_j) + C$ since $\Delta W \leq C$. It follows that if
		\begin{equation*}
			\sum_{k \ne i,j} \g(x_j - x_k) + W(x_j) \leq \sum_{k \ne i,j} \g(x_i - x_k) + W(x_i),
		\end{equation*}
		then $\Mim_{i \to j} \mcl H^\mu(Y_M) \leq \mcl H^\mu(Y_M) - \g(x_i - x_j) + C$. If the reverse inequality is true, then the same estimate holds for $\Mim_{j \to i}$ in place of $\Mim_{i \to j}$. Since $\g \geq 0$, we have proved \eref{Mim.energy}. Note that this is the only time in the course of proving \pref{main} that we use $d \geq 3$ (in $d=2$, the Coulomb kernel is $-\log |x|$ which is unbounded below).
		
		Turning to \eref{Mim.adj}, we compute by Fubini's theorem that
		\begin{align*}
			\lefteqn{ \int_{\R^d \times \R^d} f_1(y_i,y_j) \Mim_{i \to j} f_2(y_i,y_j) dy_i dy_j  } \quad & \\ &= \int_{\R^d \times \R^d} \( \int_{\R^d} f_1(y_i,y_j) dy_1 \)f_2(y_j + x, y_j) \nu(x) dx dy_j.
		\end{align*}
		Substituting $y_j + x \mapsto x$ and re-labeling variables finishes the proof.
	\end{proof}
	
	\subsection{Outline of the argument.}\ The idea of the proof is to bound $\mbb Q(E_{n,R})$ by an expression involving $\mbb Q(E_{n+1,R})$ and $\mbb Q(E_{n-1,R})$ by either transporting a particle within $B_R$ to a microscopic neighborhood of a particle outside $B_S$ or by transporting a particle outside $B_S$ to a microscopic neighborhood of a particle within $B_R$. After indexing, we will assume $y_1 \in B_R$ and $y_{n+1} \in B_S^c$, and the transport will be realized through either $\Mim_{1 \to n+1}$ or $\Mim_{n+1 \to 1}$.
	
	The transport though will not always be possible: $y_1$ or $y_{n+1}$ may lie within distance $1$ of $\pa (B_S \setminus B_R)$, and so some transported configurations will have a point within $B_S \setminus B_R$, giving them formally infinite energy by the definition of $\mbb Q$ \eref{Qdef}. Moreover, we must track entropy factors involved in the transport, and this requires working on events for which certain density bounds hold. We ignore these technicalities for now to give basic outline of the argument. All inequalities below should be read heuristically, or modulo error terms which capture possible unusual behavior of the gas.
	\begin{itemize}
		\item Let $n \geq 1$. As a first step, we must ``index" the events involved, and so for $\I\subset [M]$ we define
		\begin{equation} \label{e.EIR.def}
			E_{\I,R} = \{y_i \in B_R \ \forall i \in \I\} \cap \{y_i \not \in B_R \ \forall i \not \in \I\}.
		\end{equation} Note that $\mbb Q(E_{n,R}) = \binom{M}{n} \mbb Q(E_{[n],R})$. Combinatorial prefactors associated to indexing will play a key role in the proof.

		\item In the case that it is favorable to transport particle $y_{n+1} \not \in B_R$ to a microscopic neighborhood of particle $y_1 \in B_R$ using $\Mim_{n+1 \to 1}$, in the sense of the minimum in \eref{Mim.energy}, the energy cost is bounded by $C$. The entropy cost is bounded by the reduction in volume of phase space for particle $y_{n+1}$, and so to quantify this we choose $y_{n+1}$ within the annulus $B_{2T} \setminus B_{T}$ for $T \gg S$. Since $y_{n+1}$ occupies a microscopic neighborhood of $y_1$ after transport, the entropy cost is $T^d$. This shows, roughly speaking, that
		\begin{equation*}
			\mbb Q(E_{[n],R} \cap \{\Mim_{n+1 \to 1} \text{ favorable}\} \cap \{y_{n+1} \in B_{2T} \setminus B_T \}) \leq C T^d \mbb Q(E_{[n+1],R}).
		\end{equation*}
		Since there are order $T^d$ particles in $B_{2T} \setminus B_T$, we have
		\begin{align*}
			\lefteqn{ \mbb Q(E_{[n],R} \cap \{\Mim_{n+1 \to 1} \text{ favorable}\})} \quad & \\ &\leq \frac{CM}{T^d} \mbb Q(E_{[n],R} \cap \{\Mim_{n+1 \to 1} \text{ favorable}\} \cap \{y_{n+1} \in B_{2T} \setminus B_T \}) \\ &\leq CM \mbb Q(E_{[n+1],R}).
		\end{align*}
		
		\item In the case that it is instead favorable to transport $y_1 \in B_R$ to $y_{n+1} \not \in B_R$, we find that the reduction in volume available to $y_1$ costs a factor of order $R^d = C$. We will however record on the RHS the information that $y_1$ and $y_{n+1}$ are close together after the transport. We find
		\begin{align*}
			\lefteqn{ \mbb Q(E_{[n],R} \cap \{\Mim_{1 \to n+1} \text{ favorable}\} \cap \{y_{n+1} \in B_{2T} \setminus B_T \}) } \quad & \\ &\leq C \mbb Q(E_{\{2,\ldots,n\}} \cap \{|y_1 - y_{n+1}| \leq 1\} \cap \{y_{1} \in B_{2T+1} \setminus B_{T-1}\}).
		\end{align*}
		We can use a similar bound as before on the LHS and combine with the previous step to see
		\begin{align*}
			\mbb Q(E_{[n],R}) &\leq \mbb Q(E_{[n],R} \cap \{\Mim_{n+1 \to 1} \text{ favorable}\}) + \mbb Q(E_{[n],R} \cap \{\Mim_{1 \to n+1} \text{ favorable}\}) \\
			&\leq CM \mbb Q({E_{[n+1],R}}) + \frac{CM}{T^d} \mbb Q(E_{\{2,\ldots,n\}} \cap \{|y_1 - y_{n+1}| \leq 1\} \cap \{y_{1} \in B_{2T+1} \setminus B_{T-1}\}).
		\end{align*}
		
		\item
		We apply the above argument to particles with index $j=n+1,n+2,\ldots,M$ in place of $n+1$, and we sum the inequalities. Note that the events $E_{[n] \cup \{j\},R}$ are disjoint, and the events $\{|y_1 - y_{j}| \leq 1\}$ have typically $O(1)$ overlap since $y_{1}$ has $O(1)$ points in a microscopic neighborhood around it. This shows
		\begin{align*}
			\mbb Q(E_{[n],R}) &\leq C\mbb Q\(\bigcup_{j=n+1}^M E_{[n] \cup \{j\}, R}\) + \frac{C}{T^d} \mbb Q(E_{\{2,\ldots,n\},R} \cap \{y_{1} \in B_{2T+1} \setminus B_{T-1}\}) \\ &\leq C\mbb Q\(\bigcup_{j=n+1}^M E_{[n] \cup \{j\}, R}\) + \frac{C}{M} \mbb Q(E_{\{2,\ldots,n\},R}).
		\end{align*}
		In the last line, we used that order $T^d$ particles are in $B_{2T+1} \setminus B_{T-1}$ typically, representing a fraction of $CT^d/M$ of all particles.
		
		\item We fully ``de-index" both sides to see
		\begin{equation*}
			\mbb Q(E_{n,R})  = {\binom{M}{n}} \mbb Q(E_{[n],R}) \leq C \mbb Q(E_{n+1,R}) + C\mbb Q(E_{n-1,R}).
		\end{equation*}
		Recall that $\mbb Q$ with $M = N - X(B_S \setminus B_R)$ is the measure $\P_N$ conditioned on $X_{|B_S \setminus B_R}$. We thus have (up to error terms)
		\begin{equation*}
			\mbb P_N(E_{n,R} \cap F)  \leq C \mbb P_N(E_{n+1,R} \cap F) + C\mbb P_N(E_{n-1,R} \cap F)
		\end{equation*}
		for any Borel set $F$ which is $X_{|B_S \setminus B_R}$ measurable. We can pass to $N \to \infty$ using weak convergence to obtain the same inequality with $\P$ in place of $\P_N$ under the additional assumption that $\P(\pa F) = 0$. Using the monotone class theorem, we extend the inequality to all Borel sets $F$ which are $X_{|B_S \setminus B_R}$-measurable, which then recovers an inequality for $\P$ conditioned on $X_{|B_S \setminus B_R}$. There are error terms that we have ignored above which we partially handle by sending $T \to \infty$. Then we send $S \to \infty$ to obtain an inequality with conditioning on $X_{|B_R^c}$, and finally we reduce the remaining error terms to $\ep$ by choosing $R$ large.
	\end{itemize}
	
	Notions of typical particle density on scale $T$ and on the microscopic scale are of great importance in making the above heuristics rigorous. For the scale $T$ lower bounds on density, we will eventually use the lower bound \eref{nonzerodensity} on the density associated to $\P$ which will hold as $T \to \infty$. For the microscopic scale upper bounds on particle density, we need a quantitative result that applies directly to the conditioned measures $\mbb Q$. The following theorem is a straightforward modification of the proof of \cite[Theorem 1.1]{T23}.
	
	\begin{theorem} \label{t.LL}
		Let $\mbb Q = \mbb Q^{\mu}_{M,S,R}$ for a finite nonnegative measure $\mu$ supported within $B_S \setminus B_R$ and $T \geq 100 S$. Let $E_R$ be any event measurable with respect to the points in $B_R$, and let $U \subset B_{5T/2} \setminus B_{3T/4}$ be a Borel set. For any $\rho_0 > 0$ and $1 < r < T/100$, we have
		\begin{equation}
			\mbb Q\( \{Y(B_r(y_1)) \geq \rho r^d \}  \ \big | \ E_R \cap \{y_1 \in U \} \cap \{ Y(B_{3T} \setminus B_{T/2}) \leq  \rho_0 T^d\}\) \leq e^{-c \beta \rho^2 r^{d-2}}
		\end{equation}
		for all $\rho \geq C = C(\rho_0, \sup \Delta W,\beta^{-1})$ and a dimensional constant $c > 0$, whenever the conditional probability on the RHS exists.
	\end{theorem}
	\begin{proof}
		The result follows from the iterative procedure along balls $B_{10^k r}(y_1)$, $k=0,1,2,\ldots$, as in \cite[Proposition 2.2]{T23}. During the iteration, we transport particles within $B_{10^k r}(y_1)$ (but not $y_1$ itself) to the ball $B_{10^{k+1} r}(y_1)$. Importantly, we can terminate the iteration at a level $k$ such that $10^{dk} \rho r^d > \rho_0 T^d$ while still having $B_{10^k r}(y_1) \subset B_{3T} \setminus B_{T/2}$ as long as $\rho$ is large enough. Conditioning on points outside of $B_{3T} \setminus B_{T/2}$ has no bearing on the argument; the isotropic averaging method features perfect ``localization".
	\end{proof}
	
	\subsection{Details of the argument.}\ We give a rigorous proof following the outline. Define the events
	\begin{equation}
		\mathrm{Fav}_{i \to j} = \{ \Mim_{i \to j} \mcl H^{\mu}(Y_M) \leq \Mim_{j \to i} \mcl H^{\mu}(Y_M) \}
	\end{equation}
	and $ \mathrm{Fav}_{j \to i} = \mathrm{Fav}_{j \to i}^c$.
	The next lemma handles the case of transporting particle $n+1$ to a radius $1$ neighborhood of particle $1$.
	
	\begin{lemma} \label{l.outtoin}
		Let $A = E_{[n],R} \cap \mathrm{Fav}_{n+1 \to 1} \cap \{y_{n+1} \in B_{2T} \setminus B_T \} \cap \{y_1 \in B_{R-1}\}$. Then $\mbb Q(A) \leq CT^d \mbb Q(E_{[n+1],R})$.
	\end{lemma}
	\begin{proof}
		For some normalizing factor $\mcl Z$, we have
		\begin{align*}
			\mbb Q(A) &= \frac{1}{\mcl Z} \int_{(\R^d)^M} \1_{(B_R \cup B_S^c)^M}(Y_M) \1_A(Y_M) e^{-\beta \mcl H^\mu(Y_M)} dY_M \\
			&\leq \frac{C}{\mcl Z} \int_{(\R^d)^M} \1_{(B_R \cup B_S^c)^M}(Y_M) \1_A(Y_M) e^{-\beta \Mim_{n+1 \to 1}\mcl H^\mu(Y_M)} \\
			&\leq \frac{C}{\mcl Z} \int_{(\R^d)^M} \1_{(B_R \cup B_S^c)^M}(Y_M) \1_A(Y_M) \Mim_{n+1 \to 1} \(e^{-\beta \mcl H^\mu(\cdot)} \)(Y_M) dY_M \\
			&= \frac{C}{\mcl Z} \int_{(\R^d)^M} \Mim^\ast_{n+1 \to 1} \( \1_{(B_R \cup B_S^c)^M} \1_A \)(Y_M) e^{-\beta \mcl H^\mu(Y_M)}dY_M.
		\end{align*}
		Here, we used the inequality \eref{Mim.energy} and Jensen's inequality (since $\Mim_{i \to j}$ is given by integration against a probability measure). Note that for any $y_1,\ldots,y_n,y_{n+2},\ldots,y_M$
		$$
		\int_{\R^d} \1_A(Y_M) \1_{(B_R \cup B_S^c)^M}(Y_M) dy_{n+1} \leq CT^d \1_{B_{R-1}}(y_1) \prod_{i=2}^{n} \1_{B_R}(y_i) \prod_{i=n+2}^M \1_{B_S^c}(y_i),
		$$
		so we can compute using \eref{Mim.adj} that
		\begin{align*}
			\Mim^\ast_{n+1 \to 1} \( \1_{(B_R \cup B_S^c)^M} \1_A \)(Y_M) &\leq CT^d \1_{B_1(0)}(y_1 - y_{n+1})\1_{B_{R-1}}(y_1) \prod_{i=2}^{n} \1_{B_R}(y_i) \prod_{i=n+2}^M \1_{B_S^c}(y_i) \\
			&\leq CT^d \1_{(B_R \cup B_S^c)^M}(Y_M) \1_{E_{[n+1],R}}(Y_M).
		\end{align*}
		In the last line, we used that $y_1 \in B_{R-1}$ and $|y_1 - y_{n+1}| < 1$ implies that $y_{n+1} \in B_R$. We conclude
		\begin{equation}
			\mbb Q(A) \leq \frac{CT^d}{\mcl Z} \int_{(\R^d)^M}\1_{(B_R \cup B_S^c)^M}(Y_M) \1_{E_{[n+1],R}}(Y_M) e^{-\beta \mcl H^\mu(Y_M)} dY_M = CT^d \mbb Q(E_{[n+1],R}),
		\end{equation}
		which is the desired result.
	\end{proof}
	
	Next, we handle the transport of particle $1$ to particle $n+1$.
	\begin{lemma} \label{l.intoout}
		Let $\rho_0 > 0$ and
		$$
		A' = E_{[n],R} \cap \mathrm{Fav}_{1 \to n+1} \cap \{y_{n+1} \in B_{2T} \setminus B_T\} \cap \{ Y(B_{3T} \setminus B_{T/2} ) \leq \rho_0 T^d \}.
		$$
		Then 
		$$
		\mbb Q (A') \leq \frac{CT^d}{(M-n)(M-n+1)} \mbb Q(E_{\{2,\ldots,n\},R})
		$$
		for $C$ dependent on $\rho_0$ (and $R$ and $\delta$).
	\end{lemma}
	\begin{proof}
		Consider any fixed $Y_M$. We have (for $C = C(R)$ and $Y'_M = (y'_1,y_2,\ldots,y_M)$)
		\begin{align*}
			\lefteqn{ \int_{\R^d} \1_{(B_R \cup B_S^c)^M}(Y'_M) \1_{A'}(Y'_M) dy'_1 } \quad & \\ &\leq C \1_{B_{2T} \setminus B_T}(y_{n+1})\1_{\{Y(B_{3T} \setminus B_{T/2}) \leq \rho_0 T^d + 1\}}(Y_M)\prod_{i=2}^{n} \1_{B_R}(y_i) \prod_{i=n+2}^M \1_{B_S^c}(y_i).
		\end{align*}
		Indeed, the integrand on the LHS is only nonzero for $y'_1 \in B_R$, and one can check that if any of the indicator functions on the RHS evaluate to $0$ then so does the LHS integrand. Therefore,
		\begin{align*}
			\lefteqn{\Mim^\ast_{1 \to n+1} (\1_{(B_R \cup B_S^c)^M}\1_{A'})(Y_M)} \quad & \\ &\leq C \1_{(B_R \cup B_S^c)^M}(Y_M) \1_{E_{\{2,\ldots,n\},R}}(Y_M) \1_{\{Y(B_{3T} \setminus B_{T/2})  \leq \rho_0 T^d + 1\}}(Y_M) \\ &\quad \times \1_{B_1(0)}(y_1 - y_{n+1}) \1_{B_{2T+1} \setminus B_{T/2-1}}(y_1).
		\end{align*}
		We can apply a nearly identical computation as in the proof of \lref{outtoin}, using also \lref{Mim}, to see
		\begin{align} \label{e.aftertransport}
			 \mbb Q(A')  \leq C \mbb Q\big (& E_{\{2,\ldots,n\},R} \cap \{Y(B_{3T} \setminus B_{T/2}) \leq \rho_0 T^d +1 \} \\\notag &\quad \cap \{|y_1 - y_{n+1}| < 1 \} \cap \{y_1 \in B_{2T+1} \setminus B_{T/2-1} \} \big ).
		\end{align}
		Define 
		$$A''_i = \{y_i \in B_{2T+1} \setminus B_{T/2-1} \} \cap \{Y(B_{3T} \setminus B_{T/2}) \leq \rho_0 T^d +1 \}.$$
		We have
		$$
		\sum_{j=n+1}^M \1_{\{|y_1 - y_j| < 1\} } \leq \sum_{q = 1}^M (q-1) \1_{\{Y(B_1(y_1)) = q \}}
		$$
		pointwise in $Y_M$. Multiplying by $\1_{E_{\{2,\ldots,n\},R} \cap A''_1}$, taking expectation with respect to $\mbb Q$, and using exchangeability shows
		$$
		(M - n) \mbb Q(E_{\{2,\ldots,n\},R} \cap A''_1 \cap \{|y_1 - y_{n+1}| < 1\}) \leq \sum_{q =1}^M (q-1) \mbb Q(E_{\{2,\ldots,n\},R} \cap A''_1 \cap \{Y(B_1(y_1)) \geq q\}).
		$$
		The summand in the RHS is bounded by $e^{-c \beta q^2}\mbb Q(E_{\{2,\ldots,n\},R} \cap A''_1)$ for $q$ large enough relative to $\rho_0$ by \tref{LL}. Thus
		\begin{equation} \label{e.outsummed}
			Q(E_{\{2,\ldots,n\},R} \cap A''_1 \cap \{|y_1 - y_{n+1}| < 1\}) \leq \frac{C}{M-n} \mbb Q(E_{\{2,\ldots,n\},R} \cap A''_1).
		\end{equation}
		We have pointwise that
		$$
		\sum_{i \in \{1,n+1,\ldots,M\}} \1_{E_{\{2,\ldots,n\},R} \cap A''_i} \leq (\rho_0 T^d + 1) \1_{E_{\{2,\ldots,n\},R}}.
		$$
		Once again taking expectation and using exchangeability shows
		\begin{equation} \label{e.insummed}
			\mbb Q(E_{\{2,\ldots,n\},R} \cap A''_1) \leq \frac{\rho_0 T^d + 1}{M -n + 1} \mbb Q(E_{\{2,\ldots,n\},R}).
		\end{equation}
		Combining \eref{aftertransport}, \eref{outsummed}, and \eref{insummed} finishes the proof.
	\end{proof}
	
	\begin{proposition} \label{p.mainQ}
		For any $\rho_0 > 2$ and $\rho_0 T^d \geq 2n$, we have
		\begin{equation} \label{e.mainQ}
			\mbb Q(E_{n,R}) \leq C \(\mbb Q(E_{n+1,R}) + \mbb Q( E_{n-1,R}) \) + \mathrm{Bad}
		\end{equation}
		for
		\begin{align*}
		\mathrm{Bad} &= \mbb Q(\{Y(B_{3T} \setminus B_{T/2}) > \rho_0 T^d\}) \\ &\quad + \mbb Q(\{Y(B_{2T} \setminus B_T) < \rho_0^{-1} T^d\}) + \mbb Q(\{Y(B_{R} \setminus B_{R-1}) \geq (1-\rho_0^{-1}) n\}).
		\end{align*}
		The constant $C$ depends on $\rho_0$ but not $T$ large.
	\end{proposition}
	\begin{proof}
		We will first partition $E_{n,R}$ into a piece in which reasonable density bounds hold on certain regions and its complement. We abbreviate
		\begin{equation}
			D_1 = \{Y(B_{3T} \setminus B_{T/2}) \leq \rho_0 T^d\}, \quad D_2 = \{Y(B_{2T} \setminus B_T) \geq \rho_0^{-1} T^d\}, \quad D_3 = \{Y(B_{R-1}) \geq \rho_0^{-1} n\},
		\end{equation}
		and write
		\begin{equation} \label{e.Enr.decomp}
			\mbb Q(E_{n,R}) \leq \mbb Q(E_{n,R} \cap D_1 \cap D_2 \cap D_3) + \mathrm{Bad}.
		\end{equation}
		Here, we have bounded $\mbb Q(D_1^c) + \mbb Q(D_2^c) + \mbb Q(E_{n,R} \cap D_3^c)$ by the term $\mathrm{Bad}$. We focus on bounding the first term on the RHS. First, we introduce indices by
		$$
		\mbb Q(E_{n,R} \cap D_1 \cap D_2 \cap D_3) = \binom{M}{n} \mbb Q(E_{[n],R} \cap D_1 \cap D_2 \cap D_3).
		$$
		
		We now show that $D_2 \cap D_3$ occurring ensures there are enough index pairs $(i,j)$ with $y_i \in B_R$ and $y_j \in B_{2T} \setminus B_T$ to apply our arguments effectively. We start by noting pointwise in $Y_M$:
		$$
		\sum_{j=n+1}^M \1_{\{y_j \in B_{2T} \setminus B_T\}} \geq (\rho_0^{-1} T^d - n) \1_{D_2}.
		$$
		It follows that
		\begin{equation} \label{e.manyouter}
			\mbb Q(E_{[n],R} \cap D_1 \cap D_2 \cap D_3) \leq \frac{M - n}{\rho_0^{-1} T^d - n}\mbb Q(E_{[n],R} \cap D_1 \cap \{y_{n+1} \in B_{2T} \setminus B_T \} \cap D_3).
		\end{equation}
		Similarly, we have
		$$
		\sum_{i=1}^n \1_{\{y_i \in B_{R-1}\} \cap E_{[n],R}} \geq \rho_0^{-1} n \1_{ E_{[n],R} \cap D_3}
		$$
		and so
		\begin{equation} \label{e.manyinner}
			\mbb Q(E_{[n],R} \cap D_1 \cap \{y_{n+1} \in B_{2T} \setminus B_T \} \cap D_3) \leq C \mbb Q(E_{[n],R} \cap D_1 \cap \{y_{n+1} \in B_{2T} \setminus B_T \} \cap \{y_1 \in B_{R-1}\}).
		\end{equation}
		By \lref{outtoin} and \lref{intoout}, we may bound
		\begin{align} \label{e.applylemma}
			\lefteqn{ \mbb Q(E_{[n],R} \cap D_1 \cap \{y_{n+1} \in B_{2T} \setminus B_T \} \cap \{y_1 \in B_{R-1}\}) }\quad & \\ \notag &\leq CT^d \mbb Q(E_{[n+1],R}) + \frac{CT^d}{(M-n)(M-n+1)} \mbb Q(E_{\{2,\ldots,n\},R}).
		\end{align}
		Combining \eref{manyouter}, \eref{manyinner}, and \eref{applylemma}, we find
		\begin{align}
			\lefteqn{ \mbb Q(E_{[n],R} \cap D_1 \cap D_2 \cap D_3) } \quad & \\ \notag  &\leq \frac{C(M-n) T^d}{\rho_0^{-1} T^d - n} \mbb Q(E_{[n+1],R}) + \frac{C T^d}{(\rho_0^{-1} T^d - n)(M-n+1)} \mbb Q(E_{\{2,\ldots,n\},R}) \\ \notag &\leq CM \mbb Q(E_{[n+1],R}) + \frac{C}{M-n+1}  \mbb Q(E_{\{2,\ldots,n\},R}),
		\end{align}
		where in the last line we used $\rho_0^{-1} T^d \geq 2n$. We ``de-index" the events $E_{\I,R}$ for index sets $\I$ to see
		$$
		\mbb Q(E_{n,R} \cap D_1 \cap D_2 \cap D_3) \leq C \mbb Q(E_{n+1,R}) + C \mbb Q(E_{n-1,R}).
		$$
		Inserting this bound into \eref{Enr.decomp} finishes the proof.
	\end{proof}
	
	\begin{proof}[Proof of \pref{main}] 
		In what follows, we consider $X$ as a random variable taking values in $\mathcal{N}$, the space of locally finite, integer valued (nonnegative) measures on $\R^d$, where $X$ is distributed by either $\P$ or $\P_N$. We let $\mcl B(\mcl N)$ be the Borel $\sigma$-algebra on $\mcl N$, which is the smallest $\sigma$-algebra for which the maps $X \mapsto X(A)$ are measurable for all Borel sets $A \subset \R^d$ (c.f.\ \cite[Proposition 9.1.IV]{DVJ08}). Note that the Borel $\sigma$-algebra $\sigma(X_{|U})$ generated by $X_{|U}$ for a fixed Borel set $U \subset \R^d$ is the smallest $\sigma$-algebra for which every $X \mapsto X(A)$ is measurable for all Borel sets $A \subset U$.
		
		Recall that $\mbb Q_{M,S,R}^{X_{|B_{S} \setminus B_R}}$ is the law of $Y = X_{|B_R \cup B_S^c}$ under $\P_N$ conditioned on $X_{|B_S \setminus B_R}$ with $M = N - X(B_S \setminus B_R)$. So integration of \eref{mainQ} implies that for any $X_{|B_S \setminus B_R}$-measurable Borel set $F \subset \mcl N$, integer $n \geq 1$, and $\rho_0 > 2, T$ with $T$ large compared to $n, S,R$, we have
		\begin{align} \label{e.PNF.ineq}
			\P_N(E_{n,R} \cap F) &\leq C\P_N(E_{n+1,R} \cap F) + C\P_N(E_{n-1,R} \cap F)  \\  \notag &\quad + \P_N(D_1 \cap F) + \P_N(D_2 \cap F) + \P_N(D_3 \cap F)
		\end{align}
		for $D_1 = \{X(B_{3T}  \setminus B_{T/2}) \leq \rho_0 T^d\}$, $D_2 = \{X(B_{2T} \setminus B_T) \geq \rho_0^{-1} T^d\}$ and $D_3 = \{X(B_{R-1} )\geq \rho_0^{-1}n\}$. The constant $C$ is independent of $S$, $T$, and $N$.  Note $\P$ has an intensity dominated by Lebesgue measure on $\R^d$ by \cite[Theorem 3]{T23}, and so $\P$ does not give positive measure to the boundary of $E_{k,R}$,$D_1$,$D_2$, or $D_3$ within $\mcl N$ for any $k$. It follows that if $F$ is a continuity set for $\P$, then we have by weak convergence of $\P_N$ to $\P$ (along a subsequence) that 
		\begin{align} \label{e.PF.ineq}
			\P(E_{n,R} \cap F) &\leq C(\P(E_{n+1,R} \cap F) + \P(E_{n-1,R} \cap F)  \\  \notag &\quad + \P(D_1 \cap F) + \P(D_2 \cap F) + \P(D_3 \cap F).
		\end{align}
		The continuity sets of $\P$ include sets of the form $\{X(A) \geq k\}$ for relatively open balls $A \subset B_S \setminus B_R$ and $k \geq 0$, and the set of $X_{|B_S \setminus B_R}$-measurable continuity sets form an algebra. The set of all $X_{|B_S \setminus B_R}$-measurable sets $F$ satisfying \eref{PF.ineq} are a monotone class, so by the monotone class theorem we have that \eref{PF.ineq} holds for all $F$ in the $\sigma$-algebra generated by $\{X(A) \geq k\}$ for $A \subset B_S \setminus B_R$ a relatively open ball and $k\geq 0$. Denote by $\Sigma$ this $\sigma$-algebra. Clearly if $X \mapsto X(A)$ is $\Sigma$-measurable, then $A \subset B_S \setminus B_R$.
		
		Note that the set of relatively open balls $A \subset B_S \setminus B_R$ with rational data generate the relative topology on $B_S \setminus B_R$ and therefore generate the Borel $\sigma$-algebra (since this a countable set of generators). The set of $A$ for which the map $X \mapsto X(A)$ is $\Sigma$-measurable is itself a monotone class containing all relatively open balls in $B_R \setminus B_R$, and so the map must be $\Sigma$-measurable for all Borel sets $A \subset B_S \setminus B_R$. The Borel $\sigma$-algebra $\sigma(X_{|B_S \setminus B_R})$ is the smallest with this measurability property, so $\Sigma \supset \sigma(X_{|B_S \setminus B_R})$. It follows that \eref{PF.ineq} holds for all $X_{|B_S \setminus B_R}$-measurable sets $F$.
		
		 We conclude
		\begin{equation}
			\P(E_{n,R} \ | \ X_{|B_S \setminus B_R}) \leq C \P(E_{n+1,R} \ | \ X_{|B_S \setminus B_R}) + C \P(E_{n-1,R} \ | \ X_{|B_S \setminus B_R}) + \mathrm{Err}_1 + \mathrm{Err}_2 + \mathrm{Err}_3,
		\end{equation}
		$\P$-almost surely, where
		\begin{align*}
			\mathrm{Err}_1 &= \mbb P(\{X(B_{3T} \setminus B_{T/2}) \geq \rho_0 T^d\} \ | \ X_{|B_S \setminus B_R}), \\\mathrm{Err}_2 &= \mbb P(\{X(B_{2T} \setminus B_T) < \rho_0^{-1} T^d\}  \ | \ X_{|B_S \setminus B_R}), \\ \mathrm{Err}_3&= \mbb P(\{X(B_{R} \setminus B_{R-1}) \geq (1-\rho_0^{-1}) n\}  \ | \ X_{|B_S \setminus B_R}).
		\end{align*}
		We will now take $T \to \infty$ along a subsequence. For $\rho_0$ large enough depending on $\sup \Delta W$ and $\beta^{-1}$, we have by \cite[Theorem 1.1]{T23} that
		$$
		\mbb P_N(\{X(B_{3T} \setminus B_{T/2}) \geq \rho_0 T^d\}) \leq e^{-c \beta \rho_0^2 T^{d-2}}
		$$
		for a dimensional constant $c > 0$, uniformly in $N$. The same holds for $\P$ in place of $\P_N$ by convergence of finite dimensional distributions, and so we have
		$
		\lim_{T \to \infty} \mathrm{Err}_1 = 0
		$
		in $L^1(\P)$.
		By assumption \eref{nonzerodensity}, for $\rho_0 > \delta^{-1}$ we have $\lim_{T \to \infty} \mathrm{Err}_2 = 0$ in $L^1(\P)$ along a subsequence. We have therefore proved that $\P$ almost surely
		\begin{equation}
			\P(E_{n,R} \ | \ X_{|B_S \setminus B_R}) \leq C \P(E_{n+1,R} \ | \ X_{|B_S \setminus B_R}) + C \P(E_{n-1,R} \ | \ X_{|B_S \setminus B_R}) + \mathrm{Err}_3.
		\end{equation}
		We can now take $S \to \infty$ to see almost surely
		\begin{align} \label{e.almostdone}
			\P(E_{n,R} \ | \ X_{|B_R^c}) &\leq C \P(E_{n+1,R} \ | \ X_{|B_R^c}) + C \P(E_{n-1,R} \ | \ X_{|B_R^c}) \\ \notag &\quad + \mbb P(\{X(B_{R} \setminus B_{R-1}) \geq (1-\rho_0^{-1}) n\}  \ | \ X_{|B_R^c}).
		\end{align}
		By \cite[Theorem 3]{T23}, we have that $\mbb E_{\mbb P}[X(B_r(z))] \leq Cr^d$ for all $r > 0$. It follows by Chebyshev's inequality that for any $\ep > 0$ we have a $\rho_1$ such that
		$$
		\mbb P(\{X(B_{R} \setminus B_{R-1}) \geq \rho_1 R^{d-1}\}) \leq \ep^2/2 \quad \forall \ R \geq 1.
		$$
		Since $\rho_0 > 2$, there exists a $\sigma(X_{|B_R^c})$-measurable event $G_R^{(1)}$ with $\P(G_R^{(1)}) \geq 1 - \ep/2$ such that a.s.\ for $X \in G_R^{(1)}$ and for any $n \geq 2\rho_1 R^{d-1}$ we have
		$$
		\mbb P(\{X(B_{R} \setminus B_{R-1}) \geq (1-\rho_0^{-1}) n\}  \ | \ X_{|B_R^c}) \leq \ep.
		$$
		Inserting this into \eref{almostdone} proves \eref{main} for $n \geq 2 \rho_1 R^{d-1}$. Crucially, $\rho_1$ is chosen independent of $R > 1$.
		
		Handling the $n < 2\rho_1 R^{d-1}$ case is straightforward. By \eref{nonzerodensity}, there exists a large $R$ with $\P(\{ X(B_{R}) \geq \delta R^d\}) \geq 1 - \ep^2/2$. Therefore, we can find a $\sigma(X_{|B_R^c})$-measurable event $G_R^{(2)}$ with $\P(G_R^{(2)}) \geq 1 - \ep/2$ such that a.s.\ for $X \in G_R^{(2)}$, we have
		$
		\P(\{ X(B_{R}) < \delta R^d \} \ | \ X_{|B_R^c}) \leq \ep.
		$
		If we ensure $\delta R^d > 2\rho_1 R^{d-1}$, it follows that on this event we have $\P(E_{n,R} \ | \ X_{|B_R^c}) \leq \ep$ for all $n < 2\rho_1 R^{d-1}$. Letting $G_R = G^{(1)}_R \cap G^{(2)}_R$ finishes the proof.
	\end{proof}

	\section{Proof of \tref{infvolrhon}}
	We will prove an estimate akin to \eref{infvolrhon} for $\P_N$ and then send $N \to \infty$. We will handle the conditioning and weak limits in a similar way as in the proof of \pref{main}. For a finite, nonnegative measure $\mu$ supported within $B_R^c$, we define the probability measure
	\begin{equation} 
		\mbb Q^\mu_{n,R}(dY_n) \propto \exp(-\beta \mcl H^\mu(Y_n)) \prod_{i =1}^n \1_{B_R}(y_i) dy_i
	\end{equation}
	for point configurations $Y_n = (y_1,\ldots,y_n) \subset \R^d$, $d \geq 2$. We associate to $\mbb Q^\mu_{n,R}$ the point process $Y = \sum_{i=1}^n \delta_{y_i}$, and the energy $\mcl H^\mu(Y_n)$ is as in \eref{Hmudef} with $S = \infty$ and $M = n$. For $r > 0$, we let $\nu_r$ be the uniform probability measure on the ball $B_r(0) \subset \R^d$, and we define the isotropic averaging operator
	\begin{equation}
		\Iso_r F(Y_n) = \nu_r^{\otimes n} \ast F(Y_n)
	\end{equation}
	as a convolutional operator on nice enough functions $F$ on $(\R^d)^n$. The key observation is that, because $\mu$ is nonnegative and $\g$ is superharmonic, we have
	\begin{equation}
		\Iso_r \(\sum_{i \ne j} \g(y_i - y_j) + \int_{B_R^c} \sum_{i=1}^n \g(x-y_i) \mu(dx)\) \leq \sum_{i \ne j} \g(y_i - y_j) + \int_{B_R^c} \sum_{i=1}^n \g(x-y_i) \mu(dx).
	\end{equation}
	Moreover, because $\Delta W \leq \delta^{-1}$, we have
	\begin{equation}
		\Iso_r \(\sum_{i=1}^n W(y_i)\) \leq \sum_{i=1}^n W(y_i) + C\delta^{-1} nr^2.
	\end{equation}
	Due to the $2r$ diameter support of $\nu_r$, for points $z_i \in \R^d$ and $s_i > 0$, we have
	\begin{equation}
		\Iso_r \prod_{i=1}^n \1_{B_R \cap B_{s_i}(z_i)}(y_i) = 0 
	\end{equation}
	if $|y_i - z_i| > s_i + r$ for any $i$. Moreover, by Young's inequality, we have
	\begin{equation}
		\Iso_r \prod_{i=1}^n \1_{B_R \cap B_{s_i}(z_i)}(y_i)  \leq C\| \nu_r^{\otimes n} \|_{L^\infty} \prod_{i=1}^N s_i^{d} \leq C r^{-dn} \prod_{i=1}^N s_i^d.
	\end{equation}
	We apply the above to prove the following.
	\begin{proposition} \label{p.finvolrhon}
		For any balls $A_1,\ldots,A_n \subset B_R$, we have
		\begin{equation}
			\Q^\mu_{n,R}\(\bigcap_{i=1}^n \{y_i \in A_i\}\) \leq C r^{-dn} \prod_{i=1}^n \mathrm{vol}(A_i)
		\end{equation}
		for $r = \min_{i=1,\ldots,n} \dist(A_i,\pa B_R)$, uniformly in $\mu$.
	\end{proposition}
\begin{proof}
	We abbreviate $\Q = \Q^\mu_{n,R}$ and let $A_i = B_{s_i}(z_i)$. For a normalizing constant $\mcl Z$, we have
	\begin{equation*}
		\Q\(\bigcap_{i=1}^n\{y_i \in A_i\}\) = \frac{1}{\mcl Z} \int_{(\R^d)^n} \exp(-\beta \mcl H^\mu(Y_n)) \prod_{i=1}^n \1_{B_R \cap B_{s_i}(z_i)}(y_i) dy_i.
	\end{equation*}
	By the discussion preceding the proposition, we have $\Iso_r \mcl H^\mu(Y_n) \leq \mcl H^\mu(Y_n) + C$. By Jensen's inequality and self-adjointness of $\Iso_r$, we find (for a normalizing constant $\mcl Z$)
	\begin{align*}
		\Q\(\bigcap_{i=1}^n\{y_i \in A_i\}\) &\leq \frac{C}{\mcl Z} \int_{(\R^d)^n} \exp(-\beta (\Iso_r \mcl H^\mu)(Y_n)) \prod_{i=1}^n \1_{B_R \cap B_{s_i}(z_i)}(y_i) dy_i \\
		&\leq \frac{C}{\mcl Z} \int_{(\R^d)^n} \exp(-\beta \mcl H^\mu(Y_n)) \Iso_r \(\prod_{i=1}^n \1_{B_R \cap B_{s_i}(z_i)}(y_i) \) dy_i \\
		&\leq \frac{C \prod_{i=1}^N s_i^d}{r^{dn} \mcl Z}\int_{(\R^d)^n} \exp(-\beta \Iso_r \mcl H^\mu(Y_n)) \prod_{i=1}^n \1_{B_R}(y_i) dy_i.
	\end{align*}
	The integral on the bottom line is equal to $\mcl Z$.
\end{proof}

\begin{proof}[Proof of \tref{infvolrhon}]
	By \pref{finvolrhon}, we have for disjoint balls $A_1,\ldots,A_n \subset B_R$ that
	\begin{equation}
		\mbb Q^\mu_{n,R}\(\bigcap_{i=1}^n\{Y(A_i) \geq 1\}\) = n! \mbb Q^\mu_{n,R}\(\bigcap_{i=1}^n\{y_i \in A_i\}\) \leq C\ep^{-dn} \prod_{i=1}^n\mathrm{vol}(A_i)
	\end{equation}
	uniformly in $\mu$ for $\ep = \min_i \dist(A_i,\pa B_R)$.
	
	Note that $\mbb Q^\mu_{n,R}$ with $\mu = X_{|B_R^c}$ and $N-n = X(B_R^c)$ is the distribution of $\P_N$ conditioned on $X_{|B_R^c}$. We find $\P_N$-a.s.\ that
	\begin{equation}
		\P_N\(\{X(B_R) = n\} \cap \bigcap_{i=1}^n\{X(A_i) \geq 1\} \ \bigg | \ X_{|B_R^c}\)  \leq C \ep^{-dn} \prod_{i=1}^n \mathrm{vol}(A_i).
	\end{equation} 
	We take the limit of the above inequality as $N \to \infty$ in the same way as in the proof of \pref{main} to find that the same inequality holds for $\P$ in place of $\P_N$. Finally, we divide both sides by $\P(\{X(B_R) = n\} \ | \ X_{|B_R^c})$ to conclude. 
\end{proof}

\subsubsection*{Acknowledgments.}\ The author would like to thank Sylvia Serfaty and Thomas Lebl\'e for helpful comments. The author was partially supported by NSF grant DMS-2000205.

\addcontentsline{toc}{section}{References}
\bibliographystyle{alpha}
\bibliography{rigidbib}{}

\vskip .5cm
\noindent
\textsc{Eric Thoma}\\
Courant Institute, New York University. \\
Email: {eric.thoma@cims.nyu.edu}.
\vspace{.2cm}

\end{document}